\documentclass[toc=bibnumbered,ngerman,english,titlepage=false]{scrartcl}

\usepackage[utf8]{inputenc}
\usepackage{babel}
\usepackage[T1]{fontenc}
\usepackage{lmodern}
\usepackage{microtype}

\usepackage{amsmath,amsthm,amssymb}
\usepackage{bm}
\usepackage{textcomp}
\usepackage{csquotes}
\usepackage{enumitem}
\usepackage[color links=true,allcolors=black]{hyperref}

\usepackage{graphicx}
\usepackage{float}
\usepackage{times}

\numberwithin{equation}{section} 

\newcommand{\N}{\ensuremath{\mathbb{N}}}

\newcommand{\R}{\ensuremath{\mathbb{R}}}

\newcommand{\C}{\ensuremath{\mathbb{C}}}
\newcommand{\K}{\ensuremath{\mathbb{K}}}

\newcommand{\Ac}{\ensuremath{\mathcal{A}}}
\newcommand{\Bc}{\ensuremath{\mathcal{B}}}
\newcommand{\Cc}{\ensuremath{\mathcal{C}}}

\newcommand{\Lc}{\ensuremath{\mathcal{L}}}

\newcommand{\Pc}{\ensuremath{\mathcal{P}}}

\setkomafont{disposition}{\normalcolor\bfseries}
\setkomafont{section}{\LARGE} 
\setkomafont{subsection}{\Large} 
 
\begin{document}
\thispagestyle{plain}

\topmargin -18pt\headheight 12pt\headsep 25pt

\ifx\cs\documentclass \footheight 12pt \fi \footskip 30pt

\textheight 625pt\textwidth 431pt\columnsep 10pt\columnseprule 0pt

 \renewcommand{\headfont}{\slshape}      
 \renewcommand{\pnumfont}{\upshape}      
\setcounter{secnumdepth}{5}             
\setcounter{tocdepth}{5}             

\newtheorem{Definition}{Definition}[section]
\newtheorem{Satz}[Definition]{Satz}
\newtheorem{Lemma}[Definition]{Lemma}
\newtheorem{Korollar}[Definition]{Korollar}
\newtheorem{Corollary}[Definition]{Corollary}
\newtheorem{Bemerkung}[Definition]{Bemerkung}
\newtheorem{Remark}[Definition]{Remark}
\newtheorem{Proposition}[Definition]{Proposition}
\newtheorem{Beispiel}[Definition]{Beispiel}
\newtheorem{Theorem}[Definition]{Theorem}



\pdfbookmark[1]{Titlepage}{title}
\begin{center}
	{\LARGE Nonlocal-to-Local Convergence for a Cahn--Hilliard Tumor Growth Model\\[2ex]
	\Large\textsc{Christoph Hurm \& Maximilian Moser}
}
\end{center}

\begin{abstract}
	\noindent\textbf{Abstract.} We consider a local Cahn-Hilliard-type model for tumor growth as well as a nonlocal model where, compared to the local system, the Laplacian in the equation for the chemical potential is replaced by a nonlocal operator. The latter is defined as a convolution integral with suitable kernels parametrized by a small parameter. For sufficiently smooth bounded domains in three dimensions, we prove convergence of weak solutions of the nonlocal model towards strong solutions of the local model together with convergence rates with respect to the small parameter. The proof is done via a Gronwall-type argument and a convergence result with rates for the nonlocal integral operator towards the Laplacian due to Abels, Hurm \cite{AbelsHurm}.\\
	
	\noindent\textit{2020 Mathematics Subject Classification:} Primary 35K57; Secondary 35B40, 35K61, 35Q92.\\
	\textit{Keywords:} tumor growth; non-local and local Cahn-Hilliard equation; nonlocal to local convergence.
\end{abstract}

\pdfbookmark[1]{Table of Contents}{toc}
\tableofcontents

\section{Introduction}\label{sec_intro}
Let $\Omega\subseteq\R^n$, $n\in\{2,3\}$, be a bounded domain with $C^3$-boundary, $T>0$ be fixed and $\Omega_T:=\Omega\times(0,T)$ as well as $\partial\Omega_T:=\partial\Omega\times(0,T)$. We consider the following local Cahn-Hilliard model for tumor growth,
\begin{alignat}{2}
    \partial_t\varphi &= \Delta\mu + (\Pc\sigma - \Ac)h(\varphi) & \quad &\text{ in }\Omega_T, \label{eq_Loc1} \\
    \mu &= -\Delta\varphi + \Psi^\prime(\varphi)& \quad &\text{ in }\Omega_T, \label{eq_Loc2} \\
    \partial_t\sigma &= \Delta\sigma + \Bc(\sigma_S - \sigma) - \Cc\sigma h(\varphi)& \quad &\text{ in }\Omega_T,\label{eq_Loc3}\\
    \partial_{\mathbf{n}}\varphi &= \partial_{\mathbf{n}}\mu = \partial_{\mathbf{n}}\sigma = 0 &\quad &\text{ on }\partial\Omega_T , \label{eq_Loc4} \\
	\varphi(0) &= \varphi_0,\quad \sigma(0) = \sigma_0 & \quad &\text{ in }\Omega. \label{eq_Loc5}
\end{alignat}
Here, $\varphi:\Omega_T\rightarrow\R$ is an order parameter distinguishing the healthy tissue and tumor phase, $\mu:\Omega_T\rightarrow\R$ is the chemical potential and $\sigma:\Omega_T\rightarrow\R$ the nutrient concentration. Moreover, $\Psi:\R\rightarrow\R$ is a double well potential with wells of equal depth and minima at $\pm 1$ and $\Pc,\Ac,\Bc,\Cc\geq 0$ are constants representing tumor proliferation rate, tumor apoptosis rate, nutrient supply rate and nutrient consumption rate, respectively. Additionally, $h:\R\rightarrow[0,1]$ is an interpolation function only present in the tumor phase and the function $\sigma_S:\Omega_T\rightarrow\R$ is a preexisting nutrient concentration. Finally, $\partial_{\mathbf{n}}$ is the derivative in normal direction.

The system \eqref{eq_Loc1}-\eqref{eq_Loc5} is a special case of the models derived in Garcke, Lam, Sitka, Styles \cite{GarckeLamSitkaStyles}, in particular neglecting chemotaxis and active transport. Moreover, the system in Garcke, Lam, Rocca \cite{GarckeLamRocca} reduces to our model when setting the control to zero there and yields uniqueness and existence of strong solutions, cf.~Theorem \ref{th_WP_Loc} below. The interested reader may also consider the references in \cite{GarckeLamSitkaStyles, GarckeLamRocca} for other Cahn-Hilliard-type models for tumor growth, for example the Cahn-Hilliard-Darcy variant and optimal control problems. Let us just mention the results in \cite{Colli1, Colli2, Colli3, Frigeri, GarckeLam1, GarckeLam2} for similar systems as \eqref{eq_Loc1}-\eqref{eq_Loc5}.

Next, for $\varepsilon>0$ small we consider the nonlocal Cahn-Hilliard model for tumor growth,
\begin{alignat}{2}
    \partial_t\varphi_\varepsilon &= \Delta\mu_\varepsilon + (\Pc\sigma_\varepsilon - \Ac)h(\varphi_\varepsilon) & \quad &\text{ in }\Omega_T, \label{eq_NonLoc1} \\
    \mu_\varepsilon &= \mathcal{L}_\varepsilon\varphi_\varepsilon + \Psi^\prime(\varphi_\varepsilon) &\quad &\text{ in }\Omega_T, \label{eq_NonLoc2} \\
    \partial_t\sigma_\varepsilon &= \Delta\sigma_\varepsilon + \Bc(\sigma_S - \sigma_\varepsilon) - \Cc\sigma_\varepsilon h(\varphi_\varepsilon) & \quad &\text{ in }\Omega_T,\label{eq_NonLoc3}\\
	\partial_{\mathbf{n}}\mu_\varepsilon &= \partial_{\mathbf{n}}\sigma_\varepsilon = 0 & \quad &\text{ on }\partial\Omega_T, \label{eq_NonLoc4} \\
	\varphi_\varepsilon(0) &= \varphi_{0,\varepsilon},\quad \sigma_\varepsilon(0) = \sigma_{0,\varepsilon} & \quad & \text{ in }\Omega. \label{eq_NonLoc5}
	\end{alignat}
Here, the interpretation of the functions $\varphi_\varepsilon, \mu_\varepsilon, \sigma_\varepsilon$ and $\Psi, h, \sigma_S$ as well as the constants $\Pc,\Ac,\Bc,\Cc$ is analogous to the local model \eqref{eq_Loc1}-\eqref{eq_Loc5} above. Moreover, for $\varepsilon>0$ the nonlocal operator $\mathcal{L}_\varepsilon$ is defined by the following convolution integral,
	\begin{align}\label{eq_nonlocal_operator}
		\mathcal{L}_\varepsilon \psi(x) := \int_{\Omega}J_\varepsilon(x-y)\left(\psi(x)-\psi(y)\right)\,dy\quad\text{for all }x\in\Omega,
	\end{align}
 for integrable $\psi:\Omega\rightarrow\R$ and suitable convolution kernels $J_\varepsilon:\R^n\rightarrow\R$ such that $\mathcal{L}_\varepsilon$ approximates $-\Delta$ as 
 $\varepsilon\rightarrow 0$, cf.~Theorem \ref{th_AbelsHurm} below. In this setting, also the variational convergence 
 \begin{align*}
		\mathcal{E}_\varepsilon(\psi) := \frac{1}{4}\int_{\Omega}\int_{\Omega}J_\varepsilon(x-y)\big|\psi(x) - \psi(y)\big|^2\,dy\,dx \rightarrow \frac{1}{2}\int_\Omega|\nabla\psi|^2\,dx
	\end{align*}
as $\varepsilon\rightarrow0$ for all $\psi\in H^1(\Omega)$ is well-known, cf.~the results by Ponce \cite{Ponce1, Ponce2}. 

The nonlocal system \eqref{eq_NonLoc1}-\eqref{eq_NonLoc5} was introduced in Scarpa, Signori \cite{ScarpaSignori}, where the authors considered a more general model with chemotaxis and active transport as well as relaxation parameters. Their paper yields an existence and uniqueness result for weak solutions of \eqref{eq_NonLoc1}-\eqref{eq_NonLoc5}, cf.~Theorem \ref{th_Ex_Nonloc} below. 
The motivation to replace $-\Delta$ in \eqref{eq_Loc2} by $\Lc_\varepsilon$ in order to obtain \eqref{eq_NonLoc2} is to take into account long-range interactions, cf.~also \cite{ScarpaSignori}. Here, note that the nonlocal system \eqref{eq_NonLoc1}-\eqref{eq_NonLoc5} is of second order compared to the fourth order system \eqref{eq_Loc1}-\eqref{eq_Loc5}, hence there is no boundary condition for $\varphi_\varepsilon$ in \eqref{eq_NonLoc4}. For references in the direction of nonlocal Cahn-Hilliard-type models we refer to the references in \cite{ScarpaSignori} and Davoli et al.~\cite{DavoliRoccaScarpaTrussardi}.

 In this contribution, we apply the results in Abels, Hurm \cite{AbelsHurm}, in order to prove convergence of the weak solution of \eqref{eq_NonLoc1}-\eqref{eq_NonLoc5} to the strong solution of \eqref{eq_Loc1}-\eqref{eq_Loc5} for $\varepsilon\rightarrow 0$ together with rates of convergence with respect to $\varepsilon$. In the literature there are serveral results for nonlocal-to-local convergence. 
 For example, the convergence (without rates) of weak solutions of the nonlocal Cahn--Hilliard equation, i.e.~\eqref{eq_NonLoc1}--\eqref{eq_NonLoc5} with $\Pc,\Ac,\Bc,\Cc = 0$, to the solution of the local Cahn--Hilliard equation, i.e.~\eqref{eq_Loc1}--\eqref{eq_Loc5} with $\Pc,\Ac,\Bc,\Cc = 0$, has already been shown in various settings such as periodic boundary conditions \cite{MelchionnaRanetbauerScarpaTrussardi, DavoliRanetbauerScarpaTrussardi}, 
  Neumann boundary conditions \cite{DavoliScarpaTrussardi1, DavoliScarpaTrussardi} and degenerate mobility \cite{Elbar}. Recently, Davoli et al.~\cite{DavoliRoccaScarpaTrussardi} proved the nonlocal-to-local limit for a viscous Cahn--Hilliard model for tumor growth. The authors in \cite{AbelsHurm} also derived precise rates of convergence.

The structure of this work is as follows. In Section \ref{sec_Prelim}, we recall some preliminaries and Section \ref{sec_proof} contains the main result.

\section{Preliminaries}\label{sec_Prelim}
\subsection{Notation}
Let $\Omega\subseteq\R^n$, $n\in\N$, be open, $1\leq p \leq \infty$ and $k\in\N$. We use the notation $L^p(\Omega)$ and $W^{k,p}(\Omega)$ for the standard Lebesgue and Sobolev spaces on $\Omega$. The corresponding norms are denoted by $\|.\|_{L^p(\Omega)}$ and $\|.\|_{W^{k,p}(\Omega)}$, respectively. We also write $H^k(\Omega) := W^{k,2}(\Omega)$.

For any Banach space $X$ over $\K=\R$ or $\C$, we use the notation $X^\prime$ for its dual space. The corresponding dual pairing is denoted by $\langle .,.\rangle_{X}:X' \times X \rightarrow \K$.

Moreover, we recall that the inverse of the negative Laplacian $-\Delta_N$ with Neumann boundary condition is a well-defined isomorphism
	\begin{align}\label{eq_laplace_inverse}
		(-\Delta_N)^{-1} : \{c\in H^1(\Omega)^\prime: c_\Omega=0\} \rightarrow \{c\in H^1(\Omega): c_\Omega=0\}, 
	\end{align}
	where $c_\Omega := \frac{1}{|\Omega|}\langle c,1\rangle_{H^1(\Omega)}$ with $|\Omega|$ denoting the $n$-dimensional Lebesgue measure of $\Omega$.
\subsection{Assumptions}\label{sec_assumptions}
We make the following general assumptions.
\begin{enumerate}[label=\textnormal{(A.\arabic*)}]
    \item \label{ass_1} Let $\Omega\subseteq\mathbb{R}^n$, $n\in\{2,3\}$, be a bounded domain with $C^3$-boundary. Moreover, let $T>0$ be fixed and $\Omega_T:=\Omega\times(0,T)$ as well as $\partial\Omega_T:=\partial\Omega\times(0,T)$.
    \item \label{ass_2} 
    Let  $J_\varepsilon: \mathbb{R}^n\rightarrow[0,\infty)$ be a non-negative function given by $J_\varepsilon(x) = \frac{\rho_\varepsilon(|x|)}{|x|^2}$ for all $x\in\mathbb{R}^n$ and $J_\varepsilon\in W^{1,1}(\mathbb{R}^n)$, where $(\rho_\varepsilon)_{\varepsilon>0}$ is a family of mollifiers satisfying
	\begin{align*}
		&\rho_\varepsilon: \mathbb{R}\rightarrow [0,\infty),\quad\rho_\varepsilon\in L^1(\mathbb{R}),\quad\rho_\varepsilon(r)=\rho_\varepsilon(-r)\quad\text{for all }r\in\mathbb{R}, \,\varepsilon>0, \\
        &\rho_\varepsilon(r)=\varepsilon^{-n} \rho_1\left(\frac{r}{\varepsilon}\right)\quad\text{ for all }r\in\R,\\
		&\int_{0}^\infty\rho_\varepsilon(r)\:r^{n-1}\,dr = \frac{2}{C_n}\quad\text{for all }\varepsilon>0, \\
		&\lim\limits_{\varepsilon\searrow 0}\int_{\delta}^\infty\rho_\varepsilon(r)\:r^{n-1}\,dr = 0\quad\text{for all }\delta>0,
	\end{align*}
	where $C_n := \int_{\mathbb{S}^{n-1}}|e_1\cdot\sigma|^2\,d\mathcal{H}^{n-1}(\sigma)$.
 \item \label{ass_3} $\Pc, \Ac, \Bc,\Cc$ are non-negative constants, $\sigma_S\in L^\infty(\Omega_T)$ and $0\leq \sigma_S\leq 1$ a.e.~in $\Omega_T$.
 \item \label{ass_4} The function $h:\mathbb{R}\rightarrow[0,1]$ is of class $C^2$, bounded and Lipschitz continuous.
 \item \label{ass_5} The potential $\Psi:\mathbb{R}\rightarrow[0,\infty)$ is of class $C^3$ and satisfies
			\begin{align}
				|\Psi^\prime(s)| &\leq k_0\Psi(s) + k_1, \label{eq_psi1}\\ 
				\Psi(s) &\geq k_2|s| - k_3, \label{eq_psi2}\\ 
				|\Psi^{\prime\prime}(s)| &\leq k_4(1+|s|^2), \label{eq_psi3}\\ 
				|\Psi^\prime(s) - \Psi^\prime(t)| &\leq k_5(1+|s|^2 + |t|^2)|s-t| \label{eq_psi4}
			\end{align}
			for all $s,t\in\mathbb{R}$ and some positive constants $k_i$, $i=0,\ldots,5$.
\item \label{ass_6} 
There are constants $C_1,C_2,C_3>0$ such that for all $s\in\R$
\[
\Psi(s)\geq C_1|s|^4 -C_2 \quad\text{ and }\quad \Psi'' (s) \geq -C_3.
\] 
\end{enumerate}

\begin{Remark} \upshape
A typical example for a potential satisfying the assumptions \ref{ass_5}-\ref{ass_6} is the classical double-well potential $\Psi(s) := \frac{1}{4}(1-s^2)^2$ for all $s\in\mathbb{R}$.
\end{Remark}
\subsection{Inequalities}

\begin{Lemma}\label{th_DavoliInequ}
 	      For every $\delta>0$, there exist constants $C_\delta>0$ and $\varepsilon_\delta>0$ such that for every sequence $(f_\varepsilon)_{\varepsilon>0}\subset L^2(\Omega)$ there holds
 			\begin{align}\label{eq_ineqDav}
 				\|f_{\varepsilon_1} - f_{\varepsilon_2}\|_{L^2(\Omega)}^2 \leq \delta\mathcal{E}_{\varepsilon_1}(f_{\varepsilon_1}) + \delta\mathcal{E}_{\varepsilon_2}(f_{\varepsilon_2}) + C_\delta\|f_{\varepsilon_1} - f_{\varepsilon_2}\|_{H^1(\Omega)^\prime}^2
 			\end{align}
 		for all $\varepsilon_1,\varepsilon_2\in(0,\varepsilon_\delta)$.
 	\end{Lemma}
  \begin{proof}
      For a proof, we refer to \cite[Lemma 4(2)]{DavoliRanetbauerScarpaTrussardi}.
  \end{proof}
  
	\begin{Theorem}\label{th_AbelsHurm}
		Let $\Omega\subset\mathbb{R}^n$, $n\in\{2,3\}$, be a bounded domain with $C^3$-boundary. Moreover, let $\Lc_\varepsilon$ be defined as in \eqref{eq_nonlocal_operator} and $J_\varepsilon$ satisfy \ref{ass_2} from Section \ref{sec_assumptions}. Then for all $c\in H^3(\Omega)$ with $\partial_{\mathbf{n}}c = 0$ on $\partial\Omega$ it holds for a constant $K>0$ independent of $\varepsilon$
		\begin{align}\label{eq_AbelsHurm}
			\Big\|\mathcal{L}_\varepsilon c + \Delta c\Big\|_{L^2(\Omega)}\leq K\sqrt{\varepsilon}\|c\|_{H^3(\Omega)}. 
		\end{align}
	\end{Theorem} 
 \begin{proof}
     A proof can be found in \cite[Corollary 4.2]{AbelsHurm}.
 \end{proof}
\subsection{Existence and Uniqueness Results}
Well-posedness of the local Cahn-Hilliard model \eqref{eq_Loc1}-\eqref{eq_Loc5} is available in a slightly more general setting due to \cite{GarckeLamRocca}, where the authors considered the problem with an additional control function. 
In particular, we have the following well-posedness result for the system \eqref{eq_Loc1}--\eqref{eq_Loc5}.

\begin{Theorem}[Well-posedness of the local model]
\label{th_WP_Loc}
 Let the assumptions \ref{ass_1}, \ref{ass_3}-\ref{ass_5} hold and let $n=3$. Let the initial data $(\varphi_0,\sigma_0)$ satisfy $\varphi_0\in H^3(\Omega)$ with $\partial_\mathbf{n}\varphi_0 = 0$ on $\partial\Omega$ and $\sigma_0\in H^1(\Omega)$ with $0\leq \sigma_0 \leq 1$ a.e.~in $\Omega$.
 
 Then there is a unique solution $(\varphi, \mu, \sigma)$ of \eqref{eq_Loc1}-\eqref{eq_Loc5} with the regularity
 \begin{align*}
     \varphi &\in L^\infty(0,T,H^2(\Omega))\cap L^2(0,T,H^3(\Omega))\cap H^1(0,T,L^2(\Omega))\cap C^0(\overline{\Omega_T}),\\
     \mu&\in L^2(0,T,H^2(\Omega))\cap L^\infty(0,T,L^2(\Omega)),\\
     \sigma&\in L^\infty(0,T,H^1(\Omega))\cap L^2(0,T,H^2(\Omega))\cap H^1(0,T,L^2(\Omega)),\quad 0\leq\sigma\leq 1\text{ a.e.~in }\Omega_T,
 \end{align*}
 such that $\varphi(0) = \varphi_0$, $\sigma(0) = \sigma_0$ and for a.e.~$t\in (0,T)$ and all $\xi\in H^1(\Omega)$ it holds 
\begin{align}
    0 &= \int_{\Omega} \partial_t\varphi\xi + \nabla\mu\cdot\nabla\xi - (\Pc\sigma - \Ac)h(\varphi)\xi\,dx, \label{eq_WP_loc1}\\
    0 &= \int_{\Omega} \mu\xi - \Psi^\prime(\varphi)\xi - \nabla\varphi\cdot\nabla\xi\,dx, \label{eq_WP_loc2}\\
    0 &= \int_{\Omega}\partial_t\sigma\xi + \nabla\sigma\cdot\nabla\xi + \Cc h(\varphi)\xi + \Bc(\sigma- \sigma_S)\xi\, dx. \label{eq_WP_loc3}
\end{align}
\end{Theorem}
\begin{proof}
    We refer to \cite[Theorem 2.1]{GarckeLamRocca}.
\end{proof}

Existence of weak solutions of the nonlocal system \eqref{eq_NonLoc1}--\eqref{eq_NonLoc5} has already been shown in \cite{ScarpaSignori}, where the authors considered a more general Cahn-Hilliard system including chemotaxis and active transport, as well as possible relaxation terms for the Cahn-Hilliard equation.
The following result is obtained:

\begin{Theorem}[Well-posedness of the non-local model]
\label{th_Ex_Nonloc}
    Let the assumptions \ref{ass_1}-\ref{ass_6} hold and let $n=3$. Moreover, we assume that $\varphi_{0,\varepsilon}, \sigma_{0,\varepsilon}\in L^2(\Omega)$. Then for $\varepsilon_0>0$ small and all $\varepsilon\in(0,\varepsilon_0]$ there exists a unique weak solution $(\varphi_\varepsilon,\mu_\varepsilon,\sigma_\varepsilon)$ of \eqref{eq_NonLoc1}-\eqref{eq_NonLoc5} with the regularity
    \begin{align*}
        \varphi_\varepsilon& \in H^1(0,T,H^1(\Omega)')\cap L^2(0,T,H^1(\Omega)),\\
        \mu_\varepsilon &\in L^2(0,T,H^1(\Omega)),\\
        \sigma_\varepsilon & \in H^1(0,T,H^1(\Omega)')\cap L^2(0,T,H^1(\Omega)), \quad 0\leq\sigma_\varepsilon(t)\leq 1\text{ a.e.~in }\Omega,\text{ for all }t\in[0,T],
    \end{align*}
    such that $\varphi_\varepsilon(0)=\varphi_{0,\varepsilon}, \sigma_\varepsilon(0)=\sigma_{0,\varepsilon}$ and for all $\xi\in H^1(\Omega)$ and a.e.~$t\in(0,T)$ it holds
    \begin{align}
        0&=\langle\partial_t\varphi_\varepsilon,\xi\rangle_{H^1(\Omega)} + \int_\Omega\nabla\mu_\varepsilon\cdot\nabla\xi\,dx - \int_\Omega(\Pc\sigma_\varepsilon-\Ac)h(\varphi_\varepsilon)\xi\,dx,\label{eq_WP_Nonloc1}\\
        \mu_\varepsilon&= \Lc_\varepsilon\varphi_\varepsilon+ \Psi'(\varphi_\varepsilon),\label{eq_WP_Nonloc2}\\
        0&= \langle\partial_t\sigma_\varepsilon,\xi\rangle_{H^1(\Omega)} + \int_\Omega \nabla\sigma_\varepsilon\cdot\nabla\xi + \Cc h(\varphi_\varepsilon)\xi + \Bc(\sigma_\varepsilon- \sigma_S)\xi\, dx. \label{eq_WP_Nonloc3}
    \end{align}
\end{Theorem}

\begin{proof}
    This follows from \cite{ScarpaSignori}, Theorem 2.14 and Theorem 2.15 by setting $F_1=0$ and $\chi=0$ there. Here, $\varepsilon_0>0$ is such that for some $c_0>0$ it holds
    \begin{align*}
        \Psi^{\prime\prime}(s) +  \inf_{x\in\Omega}\int_\Omega J_\varepsilon(x-y)\,dy \geq c_0\quad\text{ for all }s\in\mathbb{R}, \varepsilon\in(0,\varepsilon_0].
    \end{align*}
    The existence of such an $\varepsilon_0>0$ follows from 
\[
\inf_{x\in\Omega}\int_\Omega J_\varepsilon(x-y)\,dy \geq \frac{c}{\varepsilon ^2}\quad\text{ for all }\varepsilon\in(0,1]
\]
for some $c>0$. The latter can be shown by using the regularity of $\partial\Omega$ to find uniform sectors pointing inside $\Omega$, by applying the transformation rule and a rescaling argument together with the properties of $J_\varepsilon$ from \ref{ass_2}.
\end{proof}

\section{Convergence Result}\label{sec_proof}

\begin{Theorem}\label{th_conv}
    Let the assumptions \ref{ass_1}--\ref{ass_6} hold, let $n=3$ and $\varepsilon_0>0$ be as in Theorem \ref{th_Ex_Nonloc}. Moreover, for the initial data $(\varphi_0,\sigma_0)$ to the local system \eqref{eq_Loc1}-\eqref{eq_Loc5} we assume $\varphi_0\in H^3(\Omega)$ with $\partial_\mathbf{n}\varphi_0 = 0$ on $\partial\Omega$ and $\sigma_0\in H^1(\Omega)$ with $0\leq \sigma_0 \leq 1$ a.e.~in $\Omega$. Additionally, for $\varepsilon\in(0,\varepsilon_0]$ let the initial data for the nonlocal system \eqref{eq_NonLoc1}-\eqref{eq_NonLoc5} satisfy $\varphi_{0,\varepsilon}, \sigma_{0,\varepsilon}\in L^2(\Omega)$ and
    \begin{align*}
        \|\varphi_{0,\varepsilon} - \varphi_0\|_{H^1(\Omega)^\prime} + \|\sigma_{0,\varepsilon} - \sigma_0\|_{L^2(\Omega)} + |(\varphi_{0,\varepsilon})_\Omega - (\varphi_0)_\Omega| \leq C\sqrt{\varepsilon}
    \end{align*}
    for some constant $C>0$ independent of $\varepsilon\in(0,\varepsilon_0]$. 
    
    Then there exists a constant $K>0$ such that the weak solution $(\varphi_\varepsilon, \sigma_\varepsilon, \mu_\varepsilon)$ of the nonlocal model \eqref{eq_NonLoc1}-\eqref{eq_NonLoc5} for $\varepsilon\in(0,\varepsilon_0]$ from Theorem \ref{th_Ex_Nonloc} and the strong solution $(\varphi,\sigma,\mu)$ of the local model \eqref{eq_Loc1}-\eqref{eq_Loc5} from Theorem \ref{th_WP_Loc} satisfy for some constant $K>0$ independent of $\varepsilon$,
    \begin{align*}
        \sup_{t\in[0,T]}\|\varphi_\varepsilon(t) - \varphi(t)\|_{H^1(\Omega)^\prime} + \|\varphi_\varepsilon - \varphi\|_{L^2(0,T;L^2(\Omega))} &\leq K\sqrt{\varepsilon}, \\
        \sup_{t\in[0,T]}\|\sigma_\varepsilon(t) - \sigma(t)\|_{L^2(\Omega)} + \|\nabla\sigma_\varepsilon - \nabla\sigma\|_{L^2(0,T;L^2(\Omega))} &\leq K\sqrt{\varepsilon}.
    \end{align*}
\end{Theorem}
\begin{Remark}\upshape
    We assumed $n=3$ in the theorem because this is also the case in \cite{GarckeLamRocca, ScarpaSignori}. However, note that Theorem \ref{th_WP_Loc}, Theorem \ref{th_Ex_Nonloc} and Theorem \ref{th_conv} should also work in the case $n=2$ because the required embeddings are improved.
\end{Remark}
\begin{proof}
    We denote by $(\varphi_\varepsilon, \mu_\varepsilon, \sigma_\varepsilon)$ the weak solution to the nonlocal model \eqref{eq_NonLoc1}-\eqref{eq_NonLoc5} given by Theorem \ref{th_Ex_Nonloc} and with $(\varphi,\mu,\sigma)$ the unique solution to the local model \eqref{eq_Loc1}-\eqref{eq_Loc5} provided by Theorem \ref{th_WP_Loc}. Then, the functions $\tilde{\varphi} := \varphi_\varepsilon-\varphi, \tilde{\mu} := \mu_\varepsilon-\mu, \tilde{\sigma} := \sigma_\varepsilon-\sigma$ solve the system
	\begin{align}
		\partial_t\tilde{\varphi} &= \Delta\tilde{\mu} + (\Pc\sigma_\varepsilon-\Ac)h(\varphi_\varepsilon) - (\Pc\sigma -\Ac)h(\varphi) &&\text{in }\Omega_T, \label{diff1} \\
		\tilde{\mu} &= \mathcal{L}_\varepsilon\varphi_\varepsilon + \Delta\varphi + \Psi^\prime(\varphi_\varepsilon) - \Psi^\prime(\varphi) &&\text{in }\Omega_T, \label{diff2} \\
		\partial_t\tilde{\sigma} &= \Delta\tilde{\sigma} + \Bc(\sigma-\sigma_\varepsilon)  - \Cc\sigma_\varepsilon h(\varphi_\varepsilon) + \Cc\sigma h(\varphi) &&\text{in }\Omega_T. \label{diff3}
	\end{align}
 in a weak sense. More precisely, the weak formulation is obtained by testing with functions in $H^1(\Omega)$, cf.~the weak formulations \eqref{eq_WP_loc1}-\eqref{eq_WP_loc3} and \eqref{eq_WP_Nonloc1}-\eqref{eq_WP_Nonloc3}.
 
 Testing \eqref{diff1} with $(-\Delta_N)^{-1}(\tilde{\varphi} - \tilde{\varphi}_\Omega)\in H^1(\Omega)$, cf.~the property \eqref{eq_laplace_inverse} for the inverse Neumann Laplacian above, we obtain
 \begin{align}\label{mean1}
			\frac{1}{2}\frac{d}{dt}\|\tilde{\varphi}&-\tilde{\varphi}_{\Omega}\|_{H^{1}(\Omega)^\prime}^2 = -\int_{\Omega}\tilde{\mu}(\tilde{\varphi}-\tilde{\varphi}_{\Omega})\,dx \nonumber\\
			&+ \int_{\Omega}\Big[(\Pc\sigma_\varepsilon-\Ac)h(\varphi_\varepsilon) - (\Pc\sigma -\Ac)h(\varphi)\Big](-\Delta_N)^{-1}(\tilde{\varphi}-\tilde{\varphi}_{\Omega})\,dx.
		\end{align}
  For the second term on the right-hand side of \eqref{mean1}, we have
			\begin{align*}
			&\int_{\Omega}\Big[(\Pc\sigma_\varepsilon-\Ac)h(\varphi_\varepsilon) - (\Pc\sigma -\Ac)h(\varphi)\Big](-\Delta_N)^{-1}(\tilde{\varphi}-\tilde{\varphi}_{\Omega})\,dx \nonumber\\
			&= \int_{\Omega}\Pc\sigma_\varepsilon(h(\varphi_\varepsilon)-h(\varphi))(-\Delta_N)^{-1}(\tilde{\varphi}-\tilde{\varphi}_{\Omega})\,dx \nonumber\\
			&-\int_{\Omega}\Ac(h(\varphi_\varepsilon)-h(\varphi))(-\Delta_N)^{-1}(\tilde{\varphi}-\tilde{\varphi}_{\Omega})\,dx \nonumber\\
			&+ \int_{\Omega}h(\varphi)\Pc\tilde{\sigma}(-\Delta_N)^{-1}(\tilde{\varphi}-\tilde{\varphi}_{\Omega})\,dx =: I_1 + I_2 +I_3.
			\end{align*}
   Using the Lipschitz continuity of $h$, cf.~assumption \ref{ass_4}, and $|\sigma_\varepsilon|\leq 1$ a.e.~in $\Omega_T$ due to Theorem \ref{th_Ex_Nonloc}, we obtain the following estimates, where $L_h$ is the Lipschitz constant of $h$:
   \begin{align}\begin{split}
       |I_1| &\leq \Pc L_h\|\tilde{\varphi}\|_{L^2(\Omega)}\|(-\Delta_N)^{-1}(\tilde{\varphi}-\tilde{\varphi}_{\Omega}))\|_{L^2(\Omega)}\\
		&\leq \tfrac{1}{36}\|\tilde{\varphi}\|_{L^2(\Omega)}^2 + \Pc K\|\tilde{\varphi}-\tilde{\varphi}_{\Omega}\|_{H^1(\Omega)^\prime}^2,  \label{calc3}\end{split}
    \\
        |I_2| &\leq \tfrac{1}{36}\|\tilde{\varphi}\|_{L^2(\Omega)}^2 + \Ac K\|\tilde{\varphi}-\tilde{\varphi}_{\Omega}\|_{H^1(\Omega)^\prime}^2,  \label{calc4} \\
        |I_3| &\leq\tfrac{1}{2}\|\tilde{\sigma}\|_{L^2(\Omega)}^2 + \Pc K\|\tilde{\varphi}-\tilde{\varphi}_{\Omega}\|_{H^{1}(\Omega)^\prime}^2. \label{calc5}
   \end{align}
   
   In the next step, we test \eqref{diff2} with $\tilde{\varphi} - \tilde{\varphi}_{\Omega}$. This yields
			\begin{align}
				\int_{\Omega}\tilde{\mu}(\tilde{\varphi} - \tilde{\varphi}_{\Omega})\,dx = \int_{\Omega}\big(\mathcal{L}_\varepsilon\varphi_\varepsilon + \Delta\varphi\big)(\tilde{\varphi} - \tilde{\varphi}_{\Omega})\,dx + \int_{\Omega}\big(\Psi^\prime(\varphi_\varepsilon) - \Psi^\prime(\varphi)\big)(\tilde{\varphi} - \tilde{\varphi}_{\Omega})\,dx. \label{mean3}
			\end{align}
   	For the first term on the right-hand side of \eqref{mean3}, it holds
			\begin{align}
				\int_{\Omega}\big(\mathcal{L}_\varepsilon\varphi_\varepsilon + \Delta\varphi\big)(\tilde{\varphi} - \tilde{\varphi}_{\Omega})\;dx = \int_{\Omega}\big(\mathcal{L}_\varepsilon\varphi+\Delta\varphi\big)(\tilde{\varphi}-\tilde{\varphi}_{\Omega})\,dx + \int_{\Omega}\mathcal{L}_\varepsilon\tilde{\varphi}(\tilde{\varphi}-\tilde{\varphi}_{\Omega})\,dx.
			\end{align}
			By symmetry of the interaction kernel $J_\varepsilon$, we observe that
			\begin{align*}
				\int_{\Omega}\mathcal{L}_\varepsilon(\tilde{\varphi}-\tilde{\varphi}_{\Omega})(\tilde{\varphi}-\tilde{\varphi}_{\Omega})\,dx 
                = 2\mathcal{E}_\varepsilon(\tilde{\varphi} - \tilde{\varphi}_\Omega).
			\end{align*}
       For the second term in \eqref{mean3}, assumption \ref{ass_5} yields
			\begin{align*}
				\int_{\Omega}\big(\Psi^\prime(\varphi_\varepsilon) - \Psi^\prime(\varphi)\big)(\tilde{\varphi} - \tilde{\varphi}_{\Omega})\,dx \geq -k_4\int_{\Omega}\tilde{\varphi}(\tilde{\varphi}-\tilde{\varphi}_{\Omega})\,dx = -k_4\|\tilde{\varphi}-\tilde{\varphi}_{\Omega}\|_{L^2(\Omega)}^2.
			\end{align*}
   
   Next, we test equation \eqref{diff3} with $\tilde{\sigma}$. Then, we obtain
	\begin{align}\label{calc9}
		\frac{1}{2}\frac{d}{dt}\|\tilde{\sigma}\|_{L^2(\Omega)}^2 + \|\nabla\tilde{\sigma}\|^2_{L^2(\Omega)} 
        &= \int_{\Omega}\big(\Bc(\sigma_S-\sigma_\varepsilon)-\Bc(\sigma_S-\sigma)\big)\tilde{\sigma}\,dx \nonumber\\
		&- \int_{\Omega}\big(\Cc\sigma_\varepsilon h(\varphi_\varepsilon) - \Cc\sigma h(\varphi)\big)\tilde{\sigma}\;dx.
	\end{align}
	For the first term on the right-hand side of \eqref{calc9}, we have
	\begin{align}\label{calc10}
		\int_{\Omega}\big(\Bc(\sigma_S-\sigma_\varepsilon)-\Bc(\sigma_S-\sigma)\big)\tilde{\sigma}\;dx = -\int_{\Omega}\Bc|\tilde{\sigma}|^2\;dx = - \Bc\|\tilde{\sigma}\|_{L^2(\Omega)}^2.
	\end{align}
	For the second term in \eqref{calc9}, we observe
	\begin{align}
		\int_{\Omega}\big(\Cc\sigma_\varepsilon h(\varphi_\varepsilon) - \Cc\sigma h(\varphi)\big)\tilde{\sigma}\,dx &= \int_{\Omega}\Cc\sigma_\varepsilon\big(h(\varphi_\varepsilon)-h(\varphi)\big)\tilde{\sigma}\,dx
		+ \int_{\Omega}\Cc h(\varphi)|\tilde{\sigma}|^2\,dx \nonumber\\
		&\leq K\Cc\|\tilde{\sigma}\|_{L^2(\Omega)}^2 + \frac{1}{36}\|\tilde{\varphi}\|_{L^2(\Omega)}^2,
	\end{align}
	where we again used $|\sigma_\varepsilon|\leq 1$ a.e.~in $\Omega_T$ and the Lipschitz continuity of $h$. Combining the previous estimates, we obtain
 	\begin{align}\label{mean4}
				&\frac{1}{2}\frac{d}{dt}\Big(\|\tilde{\varphi}-\tilde{\varphi}_{\Omega}\|_{H^{1}(\Omega)^\prime}^2 + \|\tilde{\sigma}\|_{L^2(\Omega)}^2\Big) + 2\mathcal{E}_\varepsilon(\tilde{\varphi}-\tilde{\varphi}_{\Omega}) + \|\nabla\tilde{\sigma}\|_{L^2(\Omega)}^2 + \frac{1}{2}\|\tilde{\varphi}-\tilde{\varphi}_\Omega\|_{L^2(\Omega)}^2 \nonumber\\
				&\leq \frac{1}{12}\|\tilde{\varphi}\|_{L^2(\Omega)}^2 + K\|\tilde{\sigma}\|_{L^2(\Omega)}^2 - \int_{\Omega}\big(\mathcal{L}_\varepsilon\varphi + \Delta\varphi\big)(\tilde{\varphi} - \tilde{\varphi}_{\Omega})\,dx + k_4\|\tilde{\varphi}-\tilde{\varphi}_{\Omega}\|_{L^2(\Omega)}^2 \nonumber\\
				& + K\|\tilde{\varphi}-\tilde{\varphi}_{\Omega}\|_{H^{1}(\Omega)^\prime}^2+ \frac{1}{2}\|\tilde{\varphi}-\tilde{\varphi}_\Omega\|_{L^2(\Omega)}^2.
	\end{align}

 Finally, we test \eqref{diff1} with the mean value $\tilde{\varphi}_\Omega$ and obtain
 \[
    \frac{|\Omega|}{2} \partial_t|\tilde{\varphi}_\Omega|^2 = \int_\Omega \partial_t\tilde{\varphi} \tilde{\varphi}_\Omega\,dx = \int_\Omega \left[(\Pc\sigma_\varepsilon-\Ac)h(\varphi_\varepsilon) - (\Pc\sigma -\Ac)h(\varphi)\right] \tilde{\varphi}_\Omega\,dx.
 \]
 With analogous estimates as before, it follows that
 \[
 \partial_t|\tilde{\varphi}_\Omega|^2 \leq \frac{1}{12}\|\tilde{\varphi}-\tilde{\varphi}_\Omega\|_{L^2(\Omega)}^2 + C(\|\tilde{\sigma}\|_{L^2(\Omega)}^2+|\tilde{\varphi}_\Omega|^2).
 \]
Together with estimate \eqref{mean4}, this yields 
\begin{align}
    &\frac{1}{2}\frac{d}{dt}\Big(\|\tilde{\varphi}-\tilde{\varphi}_{\Omega}\|_{H^{1}(\Omega)^\prime}^2 + \|\tilde{\sigma}\|_{L^2(\Omega)}^2 + |\tilde{\varphi}_\Omega|^2\Big) + \mathcal{E}_\varepsilon(\tilde{\varphi}-\tilde{\varphi}_{\Omega}) + \|\nabla\tilde{\sigma}\|_{L^2(\Omega)}^2 + \frac{1}{2}\|\tilde{\varphi}-\tilde{\varphi}_\Omega\|_{L^2(\Omega)}^2 \nonumber\\
    &\leq K\Big(\|\tilde{\varphi}-\tilde{\varphi}_{\Omega}\|_{H^{1}(\Omega)^\prime}^2 + \|\tilde{\sigma}\|_{L^2(\Omega)}^2 + |\tilde{\varphi}_\Omega|^2\Big) + K\|\mathcal{L}_\varepsilon\varphi + \Delta\varphi\|_{L^2(\Omega)}^2,
\end{align}
where we used inequality \eqref{eq_ineqDav} with $\delta = \frac{6}{6k_4+5}$. Finally, the Gronwall inequality and Theorem \ref{th_AbelsHurm} conclude the proof.
\end{proof}
 

\textit{Acknowledgments.} 
C.~Hurm was partially supported by the Graduiertenkolleg 2339 \textit{IntComSin} of the Deutsche Forschungsgemeinschaft  (DFG, German Research Foundation) -- Project-ID 321821685. M.~Moser has received funding from the European Research Council (ERC) under the European Union’s Horizon 2020 research and innovation programme (grant agreement No 948819). The support is gratefully acknowledged. Finally, we thank Daniel Böhme and Jonas Stange for careful proofreading. 

\setcounter{secnumdepth}{0}

\makeatletter
\renewenvironment{thebibliography}[1]
{\section{\bibname}
	\@mkboth{\MakeUppercase\bibname}{\MakeUppercase\bibname}%
	\list{\@biblabel{\@arabic\c@enumiv}}%
	{\settowidth\labelwidth{\@biblabel{#1}}%
		\leftmargin\labelwidth
		\advance\leftmargin\labelsep
		\@openbib@code
		\usecounter{enumiv}%
		\let\p@enumiv\@empty
		\renewcommand\theenumiv{\@arabic\c@enumiv}}%
	\sloppy
	\clubpenalty4000
	\@clubpenalty \clubpenalty
	\widowpenalty4000%
	\sfcode`\.\@m}
{\def\@noitemerr
	{\@latex@warning{Empty `thebibliography' environment}}%
	\endlist}
\makeatother

\footnotesize

\bibliographystyle{siam}

~\\
\textit{(C.~Hurm) Fakultät für Mathematik, Universität Regensburg, Universitätsstraße 31, D-93053 Regensburg, Germany \\ 
E-mail address:} \textsf{christoph.hurm@mathematik.uni-regensburg.de}\\
\newline
\textit{(M.~Moser) Institute of Science and Technology Austria, Am Campus 1, AT-3400 Klosterneuburg\\
E-mail address:} \textsf{maximilian.moser@ist.ac.at}

\end{document}